\documentclass{amsart}

\usepackage[utf8]{inputenc}
\usepackage{amsmath}
\usepackage{amsthm}
\usepackage{amsfonts}
\usepackage{dsfont}
\usepackage{enumitem}
\usepackage{color}
\usepackage{amssymb}
\usepackage{todonotes}
\usepackage{pst-node}
\usepackage{tikz-cd} 
\usepackage{mathtools}

\newcommand{\R}{\mathbb{R}}
\newcommand{\C}{\mathbb{C}}
\newcommand{\N}{\mathbb{N}}
\newcommand{\Z}{\mathbb{Z}}
\newcommand{\T}{\mathbb{T}}

\newcommand{\Tr}{\operatorname{Tr}}
\newcommand{\Spec}{\operatorname{Spec}}
\DeclareMathOperator{\Span}{Span}

\newtheorem{theorem}{Theorem}[section]

\newtheorem{lemma}[theorem]{Lemma}
\newtheorem{proposition}[theorem]{Proposition}
\newtheorem{example}[theorem]{Example}
\theoremstyle{definition}
\newtheorem{definition}[theorem]{Definition}
\theoremstyle{remark}
\newtheorem{remark}[theorem]{Remark}
\numberwithin{equation}{section}

\newcommand{\tr}{\Tr}

\renewcommand{\hom}{\operatorname{Hom}}

\title[]{Global symbolic calculus of pseudo-differential operators on homogeneous vector bundles}
\author{Mitsuru Wilson}

\address[Mitsuru Wilson]{%
Institute of Mathematic\\
Polish Academy of Sciences\\
Warszawa\\
Poland}
\email{mitsuru.wilson@colorado.edu}

\subjclass[2000]{}

\pagebreak

\begin{document}

\begin{abstract}
A symbolic calculus for a pseudo-differential operators acting on sections of a homogeneous vector bundle  over a compact homogeneous space $G/H$ with compact $G$ and $H$ is developed. We realize the symbol of a pseudo-differential operator as a linear operator acting on corresponding irreducible unitary representations of $H$ valued in the algebra $C^\infty(G)$ of smooth functions. We write down how left invariant vector fields of $SU(2)$ act on the sections of homogeneous vector bundles associated to the fibration $\T\hookrightarrow SU(2)\rightarrow\C P^1$, which is known as the Hopf fibration. Lastly, we outline how functional calculus of a pseudo-differential operator can be computed using our calculus.
\end{abstract}

\maketitle

\tableofcontents

\section{\bf Introduction}
The resolvent operator $(A - \lambda)^{-1}$ plays a central role in the global analysis on a compact manifold $M$ associated with a linear elliptic differential operator $A$. In particular, one can easily obtain the corresponding heat operator $e^{-tA}$ for $t\in\R_+$ from the resolvent. Detailed knowledge of the terms in the asymptotic expansions of the integral kernels of the resolvent and the heat operator is of great value in calculating the asymptotic of eigenvalues \cite{gilkey2018invariance,widom1984spectral}, scalar curvature and Ricci curvature \cite{floricel2016ricci, gilkey2018invariance} and indices of Fredholm operators \cite{atiyah1973heat, getzler1983pseudodifferential}. In mathematical physics, they are used in quantum field theory \cite{dewitt1988dynamical}.

The intrinsic symbolic calculus of pseudo-differential operators was pioneered by Widom \cite{widom1980complete} and received further contributions \cite{SFGK1988resolvent}. Its importance lies in the generality of the operators to which it may be applied. Although Widom's approach can be applied to general smooth manifolds, Ruzhansky shed light on another approach to constructing intrinsic symbolic calculus for compact Lie groups wherein examples of toruses and $SU(2)$ are explicitly computed \cite{ruzhansky2008pseudo, ruzhansky2009pseudo, ruzhansky2010quantization, ruzhansky2013global}. The basic idea of Ruzhansky was the use of the representation theory of compact Lie groups  to decompose the algebra of smooth functions on a compact Lie group into Fourier series and develop the symbolic calculus for pseudo-differential operators valued in endomorphisms of each irreducible unitary representation. The symbols act on each Fourier component as an endomorphism. 
Subsequently, parametrices for pseudo-differential operators were computed in \cite{ruzhansky2009pseudo} and a version of the local index theorem was pursued in \cite{cardona2019index}.

Smooth functions on a compact Lie group can be viewed as smooth sections of the trivial line bundle. This paper outlines how to generalize the development thereafter \cite{ruzhansky2009pseudo} to operators acting on sections of a homogeneous vector bundle.

For a compact manifold $M$, the set of H\"ormander class pseudo-differential operators on $M$ is denoted by $\Psi^m(M)$. These operators are defined to be the class of operators in $\Psi^m(\R^n)$ in all local coordinate. Operators in $\Psi^m(\R^n)$ are characterized by the
symbols satisfying
\begin{equation}
\label{eq:Hormander.pseudo.differential.operator}
\left|\partial_\xi^\alpha\partial_x^\beta p(\xi, x)\right| \leq C(1+|\xi|)^{m - |\alpha|}
\end{equation}
for all multi-indices $\alpha,~\beta$ and all $\xi,~x\in\R^n$. Moreover, the definition of H\"ormander class pseudo-differential operators extend naturally to fibre-preserving operators on sections of vector bundles, which satisfy \eqref{eq:Hormander.pseudo.differential.operator} in every locally trivial coordinate. This class of operators is studied extensively, for example, in \cite{hormander1971fourier} by H\"ormander himself. In this paper, we study classical pseudo-differential operators on homogeneous vector bundles, although much of the theory can be proved for the general H\"ormander class pseudo-differential operators. Classical pseudo-differential operators are H\"ormander class pseudo-differential operators whose symbol can be written of the form 
\[
p(\xi, x) \sim \sum_{j\geq0}p_{m-j}(\xi, x)
\]
where each $p_k\in\Psi^k(\R^n)$ for each coordinate. 
%Ruzhansky proved that his class of pseudo-differential operators is the H\"ormander class of operators.

Although their theory is mathematically intriguing, much of the differential geometric aspects are missing. This paper will explore more geometric approach than that of Ruzhansky. % and phrase the results in terms of homogeneous vector bundles. 
 Homogeneous vector bundles are quite important in mathematics. For instance, in \cite{bott1957homogeneous, bott1965index}, invariant differential operators on homogeneous vector bundles have been studied extensively and the their indices were computed in terms of the unitary representations.

In Section \ref{sec:pseudo.differential.calculus.homogeneous.bundles}, we first review the basic structure of homogeneous vector bundles. In this section, we also define the symbol of an operator, obtain a Peter-Weyl type decomposition for the space of sections of a homogeneous vector bundle and express each section as a Fourier series indexed by the unitary dual of $G$ (not the structure group $K\subset G$). 

The most important development in Section \ref{sec:pseudo.differential.calculus.homogeneous.bundles} is the composition formula \eqref{formula:composition}; in order to obtain this formula, we also give the difference operator construction \cite{fischer2015intrinsic} in the homogeneous vector bundle framework. 

In Section \ref{sec:parametrix}, we provide the parametrix formula for a classical pseudo-differential operator such that the symbol of the highest order is invertible. 

We present an example using the fibration $\T\hookrightarrow SU(2)\rightarrow\C P^1$ in Section \ref{sec:example.T.su(2)}. There, we compute how the left invariant vector fields act on the homogeneous vector bundles associated to the representations of $\T$ and their symbols. Finally, we outline how to deploy functional calculus for operators with well defined parametrix for some parameter in Section \ref{sec:remarks}.

\section*{\bf Acknowledgement}
This project has received funding from the European Union’s Horizon 2020 research and innovation programme under the Marie Skłodowska-Curie grant agreement No 691246. I thank Piotr Hajac for some fruitful discussions and supporting this project as the principal grant holder. 
I would also like to extend my appreciation to Alexander Gorokhovsky for some useful discussions directly leading to the completion of this project. 

\section{\bf Pseudo-differential calculus on homogeneous vector bundles}
\label{sec:pseudo.differential.calculus.homogeneous.bundles}

\noindent 
Symbolic calculus for pseudo-differential operators that act on smooth functions of homogeneous spaces was outlined also in \cite{ruzhansky2009pseudo}. A next generalization of this result is the symbolic calculus for pseudo-differential operators acting on sections of a homogeneous vector bundle. A section of a vector bundle $\pi:E\rightarrow M$ of rank $r$ over a smooth manifold $M$ is a function $u:M\rightarrow E$ such that $\pi(u(x))=x$ for all $x\in M$. We denote by $\Gamma(E)$ the space of all smooth sections. 

\subsection{Harmonic analysis on homogeneous vector bundle}
\label{subsec:homogeneous.vector.bundle}
We give a summary of the construction and properties of homogeneous vector bundles over $M = G/K$ where $G$ is a compact Lie group and $K$ is a closed subgroup. A vector bundle $E$ over $M$ is called a homogeneous vector bundle if $K$ acts on $E$ from the right such that 
\[
E_x \cdot k = E_{xk}~\mathrm{for}~x\in M,~k\in K
\]
and the mapping from $E_x$ to $E_{xk}$ induced by $k$ is linear for $k\in K$ and $x\in M$.

Let $(\tau,E_0)$ be a finite dimensional representation of $K$. $K$ acts on $G\times E_0$ by
$(g,v)\cdot k = (gk,\tau(k)^{-1}v)$. Then, 
\[
E = G\times_\tau E_0 := (G\times E_0)^K
\]
is is a homogeneous vector bundle, and it offers another description. 
We also say that $E$ is a homogeneous vector bundle associated to the representation $E_0$.

We further assume that $G/K$ is orientable. 
Let $E$ be a homogeneous vector bundle over $M=G/K$ associated to an irreducible unitary representation $E_0$ of $K$. 
Let $G$ act on $\Gamma(E)$ by $g\cdot u(x) = gu(g^{-1}x)$ for all $g\in G$, $u\in\Gamma(E)$, and $x\in M$. 
Then, the $G$-action on $\Gamma(E)$ extends to a unitary representation on the Hilbert space $L^2(E)$, which is obtained from $\Gamma(E)$ by completing it with respect to a $G$-invariant Hermitian inner product. Such inner product can be obtained easily. Let $\langle~,~\rangle$ be a $K$-invariant Hermitian inner product on $E_0$. Suppose $u_1,~u_2\in\Gamma(E)$ are viewed as maps from $G$ to $E_0$ such that $k^{-1}\cdot u_j(g) = u_j(gk)$. Finally, $(u_1,u_2) := \int_G\langle u_1, u_2\rangle dg$ defines a Hermitian inner product on $\Gamma(E)$ where the integration is with respect to the Haar measure. The action of $G$ with this inner product becomes unitary.

$L^2(E)$ is unitarily equivalent as a $G$-space to 
\begin{align}
\label{eq:unitary.equivalence}
L^2(G, E_0)^\tau := \left\{f\in L^2(G)\otimes E_0:  \tau(k)^{-1}f(g) = f(gk) \right\}
\end{align}
with the $G$ action given by $g_0\cdot f(g) := f(g^{-1}_0g)$. This unitary equivalence is given simply by 
\begin{align*}
A :	L^2(E)	& \longrightarrow  L^2(G, E_0)^\tau\\
	f		& \mapsto  A(f)(g) = g^{-1} f(gx),
\end{align*}

In fact, if $f:G\rightarrow \C^r$, then applying the Fourier transform to each component $f^\alpha(g)$ of $f(g)$ as a complex-valued function on $G$ expresses $f^\alpha$ as a Fourier series
\begin{align}
f^\alpha(g) = \sum_{\lambda\in\widehat G}\dim\lambda\tr(\lambda(g)\widehat {f^\alpha}(\lambda)).
\end{align}
Thus,
\begin{align}
f(g) = \sum_{\lambda\in\widehat G}\dim\lambda\tr(\lambda(g)\widehat {f}(\lambda))
\end{align}
as a vector valued Fourier transform.

This can be generalized to functions on $G$ valued in $E_0$ by taking components relative to an ordered orthonormal basis $\{e_k\}$ with respect to the Hermitian inner product $\langle~,~\rangle$. Let $\{e_k^*\}$ be the corresponding dual basis of $E^*_0$ and suppose $f\in L^2(G, E_0)^\tau$. Then, $x\mapsto e_k^*(f(x))$ is a complex valued function on $G$ and

\begin{align}
\label{eq:vector.valued.fourier.series}
f(x) = \sum_{k=1}^{\dim E_0} e_k^*(f(x))e_k = \sum_{k=1}^{\dim E_0}\sum_{\lambda\in\widehat G}\dim\lambda\tr(\lambda(g)\widehat {e_k^*(f)}(\lambda))e_k .
\end{align}
We will abuse the notation and write the Peter-Weyl decomposition of $f\in L^2(G, E_0)^\tau$ in \eqref{eq:vector.valued.fourier.series} as
\begin{align}
f(x) = \sum_{\lambda\in\widehat G}\dim\lambda\tr(\lambda(g)\widehat {f}(\lambda)),
\end{align}
which we will use frequently throughout the rest of this article. 

%%%%%%%%%%%%%%%%%%%%%%%%%%%%%%%%%%%%%
%%%%%%%% The homogeneous vector bundle symbol%%%%%%%%
%%%%%%%%%%%%%%%%%%%%%%%%%%%%%%%%%%%%%

\subsection{\bf Pseudo-differential operators and their symbols}

Let $F\rightarrow G/K$ be another homogeneous vector bundle associated to an irreducible unitary representation $F_0$ of $K$, $\{f_b\}$ be an orthonormal basis of $F_0$.  In the subsequent sections, we consider continuous linear operators
\[
A:\Gamma_c(E)\longrightarrow\mathcal{D}'\left(\Gamma(F)\right).
\]
If, in addition, $A$ satisfies $g\cdot A(f(x)) = A(g\cdot f(x))$, then it is called an invariant pseudo-differential operator.

For $i = 0, 1$ we denote by $\pi_i: M \times M \rightarrow M$ the projections $(x_0, x_1) \rightarrow x_i$. Given complex
vector bundles $E_i\rightarrow M$, $i = 0, 1$, we define the vector bundle $E_0\boxtimes E_1 \rightarrow M\times M$ by $E_0\boxtimes E_1 := \pi_0^*E_0\otimes\pi_1^*E_1$.

We have the following theorem.

\begin{theorem}[Schwartz kernel theorem]
Let $M$ be a manifold and $E,~F\rightarrow M$ vector bundles, and let
\[
\hom(E, F)\rightarrow M\times M
\]
be the bundle whose fiber at $(x, y)\in M\times M$
is $Hom(E_x, F_y)$. If
\[
A : \Gamma_c(E) \longrightarrow\mathcal{D}'\left(\Gamma(F)\right)
\]
is a continuous linear mapping, there exists {\emph Schwartz kernel distribution}
$\mathcal K_A \in\mathcal{D}'(\Gamma(\hom(E,F)))$ of $A$ such that 
\[
\langle \psi, A(\phi)\rangle 
= \langle \psi\boxtimes\phi, \mathcal K_A\rangle 
\]
where
\[
\psi\boxtimes\phi\in \Gamma^\infty(M\times M),~~~(\psi\boxtimes\phi)(x,y)=\psi(x)\phi(y).
\]
\end{theorem}

A proof is contained, for example, in \cite[Section 5.2]{hormander1971fourier}. Interpreted as a distribution, we can write the action of the operator $A$ on a section $s\in \Gamma_c(E)$ as
\begin{align}
A(s)(y) = \int_MK_A(x ,y)u(x)dx.
\end{align}
It follows from the Schwartz kernel theorem that every $G$-invariant continuous linear operator can be represented 
by convolution with an Hom$(E_0,F_0)$-valued distributional kernel $F$ on $G$ satisfying
\begin{align}
\label{eq:kernel.invariant}
F(k_1xk_2) = \tau(k_2^{-1})F(x)\tau(k_1^{-1}).
\end{align}
That is, if $g\cdot A\left(g\cdot u(x)\right) = A(u(x))$ is a continuous operator, then
\begin{align*}
A u(x) = \int_GF (y^{-1}x)u(y)dy,\quad u \in D_\tau (G, E_0)
\end{align*}
for some $F\in \mathcal D'_{\tau_{E_0},\tau_{F_0}}\left(G,~\operatorname{Hom}(E_0,F_0)\right)$ where $D'_{\tau_{E_0},\tau_{F_0}}\left(G,~\operatorname{Hom}(E_0,F_0)\right)$ consists of elements that satisfy \eqref{eq:kernel.invariant}. From now on, we study exclusively the $G$-invariant continuous linear operators 
$A:\Gamma(E)\rightarrow \mathcal D'\left(\Gamma(F)\right)$. We will also simply write $\Gamma(E)$ in place for $\Gamma_c(E)$ since the base space 
$G/K$ is always compact in this article. $G$-invariant continuous linear operators are also called homogeneous operators.

Let $\lambda : G \rightarrow U\left(\mathcal H_\lambda \right)$ be an irreducible unitary representation, $E$ and $F$ homogeneous vector bundles associated to irreducible unitary representations $E_0$ and $F_0$ of a closed subgroup $K\subset G$, respectively. 
The symbol of a homogeneous operator $A :\Gamma(E) \rightarrow\Gamma(F)$ at $x \in G$ and $\lambda \in \operatorname{Rep}(G)$ is defined as 
\begin{align}
\label{formula:symbol.definition}
\sigma_A(\lambda, x) :=\widehat k_x(\lambda) \in \operatorname{End}\left(\mathcal{H}_{\lambda}\right)\otimes\hom(E_x,F_x)
\end{align}
where $k_y(x) = K_A(x, y)$ is the Schwartz kernel of $A$. Hence, 
$$
\sigma_A(\lambda, x) = \int_G \lambda(x)^* K_A(x, y) dx
$$
in the sense of distributions, and the Schwartz kernel 
can be regained from the symbol as well: 
\[
K_A(x, y) = \sum_{\lambda \in \widehat G} \dim(\lambda) 
\Tr\left(\lambda(y) \sigma_A(\lambda, x)\right)
\]
where this equality is interpreted in the sense of distribution and the trace is taken over the indices of $\lambda$. 
The product $\lambda(y) \sigma_A(\lambda, x)$ is interpreted as $\lambda(y)$ being multiplied on each 
$\operatorname{End}\left(\mathcal{H}_{\lambda}\right)$-component of $\sigma_A(\lambda, x)$. 
The following proposition shows that a homogeneous operator $A$ can be represented by its symbol.

\begin{proposition}
Let $\sigma_A$ be the symbol of a homogeneous operator 
$A : \Gamma(E) \rightarrow \Gamma(F)$. Then

\begin{align}
\label{eq:representation.of.pseudo}
A u(x) = \sum_{\lambda \in \widehat G} \dim\lambda \Tr\left(\lambda(x) \sigma_A(\lambda,x) (\widehat u (\lambda))\right)
\end{align}
for every $u \in \Gamma(E)$ and $x \in G$.
\end{proposition}

\begin{proof}
%By \eqref{eq:vector.valued.fourier.series}, 
Since $u$ can be represented as a $K$-invariant $E_0$-valued function on $G$, and $\sigma_A(\lambda,x)\in \operatorname{End}\left(\mathcal{H}_{\lambda}\right)\otimes\hom(E_x,F_x)$ acts on $\widehat u (\lambda)\in \operatorname{End}\left(\mathcal{H}_{\lambda}\right)\otimes E_0$. 
The case of $E_0=\C$, the trivial $K$-representation, with $K=\{e\}$, the trivial group, has been treated in \cite[Theorem 2.4]{ruzhansky2008pseudo}. 

Again, since $\Phi(\cdot):=\sigma_A(\lambda,\cdot)$ is an element in $\operatorname{End}\left(\mathcal{H}_{\lambda}\right)\otimes\Gamma(\hom(E,F))$, repeating the argument used for \eqref{eq:unitary.equivalence}, 
\begin{align}
\Phi\in\operatorname{End}\left(\mathcal{H}_{\lambda}\right)\otimes \big(C(G)\otimes\hom(E_0,F_0)\big)^\tau.
\end{align}
Since $\hom(E_0,F_0)$ is a finite dimensional vector space, the analysis in \cite[Theorem 2.4]{ruzhansky2008pseudo} would be practically unaffected. This completes the proof.
\end{proof}

\begin{definition}
For a symbol $\sigma_A$, the corresponding operator $A$ defined by \eqref{eq:representation.of.pseudo} will be also denoted by $\operatorname{Op}(\sigma_A)$. The operator defined by formula \eqref{eq:representation.of.pseudo} 
will be called the pseudo-differential operator associated to the symbol $\sigma_A$.
\end{definition}

Note again that, by \eqref{formula:symbol.definition}, $\sigma_A$ carries two sets of indices from $\operatorname{Hom}\big(E_x,F_x\big)$ and $\operatorname{End}(\mathcal H_\lambda)$. 
Elements in $\hom(E_x,F_x)$ can be represented as matrices relative to bases of $E_x$ and $F_x$. The following is a straightforward generalization of \cite[Theorem 10.4.6]{ruzhansky2009pseudo}.

\begin{proposition}
Suppose $\{e_a\}$ and $\{f_b\}$ be orthonormal bases of $E_0$ and $F_0$, respectively. Let  $\sigma_A(\lambda,x)$ be the symbol of a continuous operator
$A:\Gamma(E)\rightarrow \Gamma(F)$.
Then,
\begin{align}
\label{formula:symbol}
\sigma_A(\lambda, x)_{ab} = \big(\left(\lambda(x)^*\otimes f^*_b\right)A\left(\lambda(x)\otimes e_a\right)\big).
\end{align}
\end{proposition}

\begin{proof}
For each $a$ and $b$, and $m$ and $n$,
\begin{align*}
\sum_{\alpha=1}^{\dim\lambda}\lambda_{\alpha m}^*(x) f_b^*A(\lambda_{\alpha n} e_a)
&=\sum_{k=1}^{\dim\lambda}\lambda_{\alpha m}^*(x) f_b^*\sum_{\eta\in\hat{G}}\dim\eta\Tr(\eta(x)\sigma_A(\eta,x)(\widehat{\lambda_{\alpha n}}(\eta)e_a))\\
&=\sum_{k=1}^{\dim\lambda}\lambda_{\alpha m}^*(x)f_b^*\sum_{\eta\in\hat{G}}\dim\eta\sum_{i,j,\ell}\eta_{ij}(x)\sigma_A(\eta_{j\ell},x)\left(\widehat{\lambda_{kn}}(\eta)_{\ell i}e_a\right)\\
&=\sum_{k,j}\lambda_{\alpha m}^*(x)f_b^*\left(\dim\eta~\lambda_{kj}(x)\sigma_A(\lambda_{jn},x)(e_a)\frac{1}{\dim\eta}\right)\\
&=\sum_{k,j}\lambda_{\alpha m}^*(x)\lambda_{kj}(x)f_b^*\left(\sigma_A(\lambda_{jn},x)(e_a)\right)\\
&=f_b^*\left(\sigma_A(\lambda_{mn},x)(e_a)\right)\\
&=\sigma_A(\lambda_{mn},x)_{ab}
\end{align*}
where $\sigma_A(\lambda_{mn},x)$ simply means $m,n$ component of the $\operatorname{End}(\mathcal H_\lambda)$-indices.
\end{proof}

The representation \eqref{formula:symbol} of the symbol is relative to the choices of bases. The change of symbol under changes of bases is given simply by the change of basis theorem in linear algebra. For a completion, we state this fact with our notations.

\begin{proposition}[change of basis]
Suppose $\{e_a\}$ and $\{f_b\}$ be orthonormal bases of $E_0$ and $F_0$, respectively, and $\{e_a'\}$ and $\{f_b'\}$ be other orthonormal bases. Let  $\sigma_A(\lambda,x)$ be the symbol of a continuous operator
$A:\Gamma(E)\rightarrow \Gamma(F)$ relative to the first set of first bases and $\sigma_A'(\lambda,x)$ be the symbol of $A$ relative to the second set of bases.
Then, 
\begin{align}
\label{formula:change.of.basis}
\sigma_A'(\lambda, x) = \left(1\otimes V\right)\circ \sigma_A(\lambda, x)\circ\left(1\otimes U^*\right)
\end{align}
where $U:E_0\rightarrow E_0$ and $V:F_0\rightarrow F_0$ are the change of basis linear transformations $e_a\mapsto e'_a$ and $f_b\mapsto f'_b$, respectively. 
\end{proposition}

Note that using the identification $\Gamma\left(E\right)\cong \big(L^2(G)\otimes E_0\big)^K$ in \eqref{eq:unitary.equivalence}, $f\in\big(L^2(G)\otimes E_0\big)^K$ carries induced action of the Lie algebra $\mathfrak{g}$ of $G$ given by the usual formula:
\begin{align}
X\cdot f(g) := \frac{d}{dt}\Big|_{t=0}f(g\exp(tX)),\qquad g\in G,~X\in\mathfrak{g},~u\in\big(L^2(G)\otimes E_0\big)^K.
\end{align}

This action of the Lie algebra $\mathfrak g$ extends to an action of the universal enveloping algebra $U(\mathfrak g)$ of $\mathfrak g$. 
Consider the action of the operator
\begin{equation}
\label{eq:laplacian}
\mathcal L := - X_1^2 - \ldots-X_{\dim G}^2
\end{equation}
where $\left\{X_1, \ldots, X_{\dim G}\right\}$ is  an orthonormal basis of $\mathfrak g$ with respect to some bi-invariant bilinear form $B(\cdot,\cdot)$. $\mathcal L$ does not depend on a choice of orthonormal bases and $\mathcal L(\lambda_{jk}(x)) = c_\lambda\lambda_{jk}(x)$ for some $c_\lambda\in\C$, which depends only on the class of the unitary irreducible representation $\lambda\in\widehat G$. We denote by $\langle\lambda\rangle:=\left(1 + |c_\lambda|^2\right)^{1/2}$.

%The Laplacian $\mathcal L$ is symmetric and $I - \mathcal L$ is positive. 
Denote $\Xi=(I + \mathcal L)^{1/2}$. Then, $\Xi^s\in\operatorname{End}\big(\Gamma(E)\big)$ and $\Xi^s\in\mathcal D'\left(\Gamma\left(\operatorname{End}(E)\right)\right)$ for every $s\in\R$. Let us define 
\begin{align}
\label{eq:sobolev.space}
\langle u, v\rangle_s := \left(\Xi^s u, \Xi^s v\right)_{L^2(E)}
\end{align}
for $u, v \in\Gamma(E)$.
The completion of $\Gamma(E)$ with respect to the norm 
$u \mapsto\Vert u\Vert_s=\langle u, u\rangle_s^{1/2}$ lends us a definition of the Sobolev space $H^s(E)$ of order $s \in\R$. It is easy to check that the operator $\Xi^r$ defines an isomorphism $H^s(E) \rightarrow H^{s - r}(E)$ for every $r, s \in \R$. 

In view of the identification \eqref{eq:unitary.equivalence}, it can be proved that $H^s(E)$ is unitarily equivalent to $(H^s(G)\otimes E_0)^K$.  Note that $\Xi^s\in\Psi^s(E)$. Using this identification,

\begin{proposition}[Symbols for some operators]
Let $A$ be a pseudo-differential operator, $X\in\mathfrak g$ and $\phi\in C^\infty(G)$. Then,
\begin{align} 
\sigma_{\phi A}(\lambda, x) & = \phi(x) \sigma_A(\lambda, x) \\ 
%\sigma_{A \circ X}(\lambda, x) & = \sigma_A(\lambda, x) \sigma_X(\lambda, x) \\ 
\sigma_{X\circ A}(\lambda, x) &=\sigma_X(\lambda, x) \sigma_A(\lambda, x)+\left(X\sigma_A\right)(\lambda, x) 
\end{align}
\end{proposition}

		\begin{proof}
Suppose $A$ is an operator and $\phi\in C^\infty(G)$, $e_1,\ldots,e_{\dim E_0}$ and $f_1,\ldots,f_{\dim F_0}$ be bases of $E_0$ and $F_0$. Then, since
	\begin{align*}
\lambda^*(x)\otimes f^*_b\left(\phi(x) A \lambda(x)\otimes e_a\right)
& = \phi(x)\lambda^*(x)\otimes f^*_b\left(A \lambda(x)\otimes e_a\right),
	\end{align*}
$\sigma_{\phi(x)A}(\lambda,x) = \phi(x)\sigma_A(\lambda,x)$.

%Similarly, if $X\in\mathfrak g$, we have $\sigma_{A \circ X}(\lambda, x) = \sigma_A(\lambda, x) \sigma_X(\lambda, x)$,
Since
$X\big(\lambda(x)\otimes\sigma_A(\lambda,x)\big) = X\big(\lambda(x)\big)\otimes\sigma_A(\lambda,x) + \lambda(x)\otimes X\big(\sigma_A(\lambda,x)\big)$, 
	\begin{align*}
X \circ A f(x) 
& = X \sum_{\lambda \in \widehat G} \dim\lambda \Tr\left(\lambda(x)\otimes\sigma_A(\lambda, x) \widehat f(\lambda)\right) \\
& = \sum_{\lambda\in \widehat G} \dim\lambda \Tr\left(\left(X \lambda\right)(x)\otimes\sigma_A(\lambda, x) \widehat f(\lambda)\right) \\ 
& \qquad+\sum_{\lambda\in \widehat G} \dim\lambda \Tr\left(\lambda(x)\otimes X\left(\sigma_A(\lambda, x)\right) \widehat f(\lambda)\right) \\
& = \sum_{\lambda\in \widehat G} \dim\lambda \Tr\left(\lambda(x)\lambda^*(x)X\left(\lambda\right)(x)\otimes\sigma_A(\lambda, x) \widehat f(\lambda)\right) \\ 
& \qquad+\sum_{\lambda\in \widehat G} \dim\lambda \Tr\left(\lambda(x)\otimes X\left(\sigma_A(\lambda, x)\right) \widehat f(\lambda)\right) 
	\end{align*}
		\end{proof}

Suppose $G\subset U(n)$ is a connected compact matrix Lie group. Then, the symbol $\sigma_{X}$ of a left invariant vector field $X\in\mathfrak g$ can be computed relatively simply. Let $\exp:\mathfrak g\rightarrow G$ be the exponential map and $\lambda:G\rightarrow$U$(n)$ a representation so that 
		\begin{align}
	\nonumber
\lambda(g)^*X\left(\lambda(g)\right)	& = \frac{d}{dt}\Big\vert_{t=0}\lambda(g)^*\lambda\left(g\exp(tX)\right)\\
	\nonumber
							& = \frac{d}{dt}\Big\vert_{t=0}\lambda\left(\exp(tX)\right)\\
	\label{eq:symbol.of.vector.fields}
							& = \lambda_*X
		\end{align}
at the identity of $e\in G$. In fact, $\lambda(e)=I_{\dim\lambda}$, so the symbol is essentially $X$ put in block diagonal form. In particular, if $G=SU(2)$, the inclusion $\lambda:SU(2)\rightarrow U(2)$ is the unique fundamental representation with $V_\lambda = \C^2$. A basis of the complexification $\mathfrak{su}(2)_\C$ of $\mathfrak{su}(2)$ is given by
\begin{align}
\label{eq:pauli.matrices}
H = 
\begin{pmatrix}
1 & 0\\
0 & -1
\end{pmatrix}, \qquad
X = 
\begin{pmatrix}
0 & 1\\
0 & 0
\end{pmatrix}, \qquad
Y = 
\begin{pmatrix}
0 & 0\\
1 & 0
\end{pmatrix}
\end{align}
and a generic element of $SU(2)$ can be written as 
\[
\begin{pmatrix}
\alpha & -\bar \beta\\
\beta & \bar\alpha
\end{pmatrix},\qquad \vert\alpha\vert^2 + \vert\beta\vert^2 = 1.
\]
Thus, 
 $\sigma_A(\lambda,x) = A$ for all $A\in\mathfrak{su}(2)$. In general, the representation of $SU(2)$ is given by the symmetric tensor power of $\C^2$, so it suffices to compute the symbol of left invariant vector fields at the fundamental representation. That is,
\begin{align}
\label{eq:representation.of.su(2)}
\sigma_A(\operatorname{Sym}^k(\lambda),x)	& = \operatorname{Sym}^k(\lambda)_*(A) \\
\nonumber
									& = \Lambda^k(\lambda_*)(A)
\end{align}
where $\Lambda^k(\lambda_*):\mathfrak{su}(2) \rightarrow \mathfrak{u}\big(\operatorname{End}(\Lambda^k(\C^2))\big)$ is the $k$th antisymmetric tensor power representation, which is known to be irreducible unitary and these are all the irreducible unitary representations of $\mathfrak{su}(2)$.

From the general representation theoretic viewpoint, every irreducible unitary representation of a connected compact Lie group is contained in the tensor products of the fundamental representations. We will use this example in Section \ref{sec:example.T.su(2)}.

As had been studied in literature \cite{SFGK1988resolvent, widom1980complete}, the formula which highlights the study of pseudo-differential calculus is the composition formula for the symbols. 
Let $H_0$ be yet another irreducible unitary representation of $K$ and $H$ the associated vector bundle. 
Consider 
\begin{align}
\Gamma(E)\overset{A}{\longrightarrow}\Gamma(F)\overset{B}{\longrightarrow} \Gamma(H)
\end{align}
where the domains and the ranges of $A$ and $B$ makes sense.
It is natural to question how to express the symbol $\sigma_{AB}$ of the composed operator $AB$ in terms of individual symbols $\sigma_A$ and $\sigma_{B}$. In the alternative formulation by Widom, the symbol of an operator is some function in two continuous variables and the symbol of the composition of two operators is expressed as a linear combination of products of derivatives of each symbol. However, in our setting, we would have to reformulate the notion of differentiability in the representation variable, which will be done in the next subsection.

\subsection{Difference operators}

First, we recall some results from \cite{fischer2015intrinsic}.
A symbol can be viewed as a collection $\sigma = \{\sigma(\lambda,x)\in\operatorname{End}(\mathcal{H}_{\lambda})\otimes\hom(E_x,F_x): \lambda \in \widehat{G},x\in G\}$. Moreover, a norm on symbol can be defined using the Hilbert-Schmit inner product:

\begin{align}
\langle\sigma,\tau\rangle = \Tr\left(\tau^*\sigma\right)
\end{align}
where $\Tr(\cdot)$ above is the matrix trace. The operator norm $\Vert\cdot\Vert_{op}$ of an $m\times n$ matrix $A$ is defined as 
\begin{equation}
\label{eq:operator.norm.of.symbols}
\|A\|_{op} :=\sup\left\{\|A x\| : x \in \C^n,\|x\|\leq 1\right\}.
\end{equation}

Note that each $\psi\in C^\infty(G)$ defines a left convolution $L(\psi)$ and a right convolution $R(\psi)$, which act on on $\Gamma(E)$:
\begin{align}
L(\psi)(f)(x) = \psi*f(x) := \int_G \psi(y)f(y^{-1}x)dy
%\qquad \mathrm{and}
%\qquad
\end{align}
and
\begin{align}
R(\psi)(f)(x) = f*\psi(x):= \int_G f(xy) \psi(y) dy.
\end{align}

It is easy to show that 
\begin{equation}
\|L(\psi)\|_{B\left(L^2(E)\right)} 
= \|R(\psi)\|_{B\left(L^2(E)\right)}
= \sup _{\lambda \in \operatorname{Rep}(G)}\|\widehat{\psi}(\lambda)\|_{op}.
\end{equation}

On the other hand, the operator associated to $\sigma$ is the operator $\operatorname{Op}(\sigma)$ defined on $L^2(E)$ by 
\begin{align}
\operatorname{Op}(\sigma)\left(u(x)\right) = \sum_{\lambda \in \widehat G} \dim\lambda\Tr(\lambda(x) \sigma(\lambda, x) \widehat{u}(\lambda)), \quad u \in L^2(E), x \in G.
\end{align}

We also recall the notion of difference operators and of classes of symbols in \cite{fischer2015intrinsic}, which 
a refinement of its initial discovery made in \cite{ruzhansky2008pseudo}. For each $\lambda, \mu \in \operatorname{Rep}(G)$ and $\sigma \in \Sigma(G)$, we define the linear mapping 
$\triangle_\lambda \sigma(\mu)$ on $\mathcal{H}_\lambda\otimes\mathcal{H}_{\mu}$  by
\begin{align}
\triangle_{\lambda} \sigma(\mu) :=\sigma(\lambda \otimes \mu) - \sigma\left(I_{\mathcal H_\lambda} \otimes \mu\right).
\end{align}

We also define the iterated difference operators as follows. For any $a \in \N$ and $\tau=\left(\lambda_1, \ldots, \lambda_a\right) \in \operatorname{Rep}(G)^a$, we write 
\begin{align}
\triangle^\tau :=\triangle_{\lambda_1} \ldots \triangle_{\lambda_a}.
\end{align}

If $\xi \in \operatorname{Rep}(G)$ and $\sigma \in \Sigma$, then $\triangle^\tau\sigma(\xi,x)$ is an element of 
\begin{align}
\label{space:endmorphism.tensor.representations}
\operatorname{End}\left(\mathcal H_\xi^{\otimes\tau}\right)\otimes&\hom(E_x,F_x) \\
\nonumber
&:= \operatorname{End}\left(\mathcal H_{\lambda_1} \otimes \ldots \otimes \mathcal H_{\lambda_a} \otimes \mathcal H_\xi \right)\otimes \hom(E_x,F_x).
\end{align}

\begin{definition}[\cite{fischer2015intrinsic}]
Let $m \in \R$. The set $S^m(E,F)$ is the space of all the symbols $\sigma = \{\sigma(\lambda, x)\in \operatorname{End}\left(\mathcal H_\lambda\right)\otimes \hom(E_x,F_x):(\lambda, x) \in \widehat G \times G\}$, which are smooth in $x$ such that for each $\tau \in \operatorname{Rep}(G)^a$ and $X \in \operatorname{Diff^k}(G)$ there exists $C>0$ satisfying 
\begin{align}
\forall(\lambda, x) \in\widehat G\times G \qquad\left\Vert X \triangle^\tau\sigma(\lambda, x)\right\Vert_{op}\leq C\langle\lambda\rangle^{\frac{m - a}{2}}.
\end{align}
The norm $\Vert\cdot\Vert_{op}$ is in the sense of the operator norm as in \eqref{eq:operator.norm.of.symbols}.
We say that a symbol is smoothing when it is in 
\begin{align}
S^{-\infty}(E,F)=\bigcap_{m \in \mathbb{R}} S^{m}(E,F).
\end{align}
\end{definition}
Now we are ready to move onto the composition formula in the next subsection.

\subsection{Composition formula}

In general, we can prove the following composition theorem.

\begin{theorem}[Composition formula]
\label{theorem:composition.formula}
Let $m_1, m_2 \in \R$ and $\rho>\delta \geq 0$.
Let $E_0$, $F_0$ and $H_0$ be irreducible unitary representations of $K$ and $E$, $F$ and $H$ be corresponding associated vector bundles to the principal bundle $G\rightarrow G/K$. Suppose
\begin{align}
A:\Gamma(E)\longrightarrow \Gamma(F)
\end{align}
and 
\begin{align}
B:\Gamma(F)\longrightarrow\Gamma(H)
\end{align}
are continuous linear maps such that $A\left(\Gamma(E)\right) \subset \Gamma(F)$ with symbols $\sigma_A$ and $\sigma_B$ and they satisfy
\begin{align}
\left\|\triangle_\lambda^\alpha\sigma_A(\lambda, x)\right\|_{op} & \leq C_\alpha\langle\lambda\rangle^{m_1 - \rho|\alpha|} \\
\left\|X^\beta\sigma_B(\lambda, x)\right\|_{op} & \leq C_\beta\langle\lambda\rangle^{m_2 + \delta|\beta|}.
\end{align} 
for all multi-indices $\alpha$ and $\beta$, uniformly in $x\in G$ and $\lambda\in\widehat G$. Then,
\begin{align}
\label{formula:composition}
\sigma_{A B}(x, \lambda) \sim \sum_{\alpha \geq 0} \frac{1}{\alpha !}\left(\triangle_{\lambda}^\alpha \sigma_A\right)(x, \lambda) X^{(\alpha)} \sigma_B(x, \lambda).
\end{align}
\end{theorem}

\begin{proof}

In view of \cite[Theorem 8.3]{ruzhansky2008pseudo}, it suffices to reduce this case to the case of scalar functions. This is achieved by representing $f\in \Gamma(E)$ as an element in $E_0$-valued function on $G$. Again, the the scalar valued (trivial line bundle) case of this formula had been proved in \cite[Theorem 8.3]{ruzhansky2008pseudo}, which proves it for each component of $f$. Since the fibre $E_0$ is finite dimensional, that result extends trivially.
\end{proof}

\begin{remark}
Theorem \ref{theorem:composition.formula} is a slight improvement of the treatment in \cite[Chapter 13]{ruzhansky2009pseudo} for the case of scalar valued functions on homogeneous spaces because our case for $H=\{e\}$ and $E_0=\C$ the trivial representation would reduce to their case.
\end{remark}

\subsection{Sobolev spaces}

Recall that Sobolev spaces were defined in Section \ref{subsec:homogeneous.vector.bundle} by completion of $\Gamma(E)$ with the inner product in \eqref{eq:sobolev.space}.

\begin{theorem}
\label{theorem:boundedness.in.sobolev.space}
Let $G$ be a compact Lie group, $\mathfrak g$ its Lie algebra and $A:\Gamma(E)\rightarrow\Gamma(F)$ be an operator with symbol $\sigma_A$. 
Suppose that there are constants $m, C_\alpha\in\R$ such that 
\begin{align}
\label{inequality:sobolev.space}
\left\Vert X^\alpha\sigma_A(\lambda, x)\right\Vert\leq C_\alpha\langle\lambda\rangle^m
\end{align}
holds for all $x \in G$, $\lambda\in$Rep$(G)$ and all multi-indices $\alpha$ where $X^\alpha = X_1^{\alpha_1} \cdots X_{\dim G}^{\alpha_{\dim G}}$ and 
$\left\{X_1, \ldots, X_{\dim G}\right\}$ is a basis of $\mathfrak g$. Then, $A$ is a bounded linear operator 
\begin{align}
H^s(E) \longrightarrow H^{s - m}(F)
\end{align}
for all $s \in \R$.
\end{theorem}

\begin{proof}
The case of $H=\{e\}$ and $E_0=F_0=\C$ was shown in \cite[Theorem 3.2]{ruzhansky2013global}. Suppose $\{e_k\}$ and $\{f_\ell\}$ are bases of $E_0$ and $F_0$, respectively, with corresponding dual bases $\{e^*_k\}$ and $\{f^*_\ell\}$. Note that 
$B(e^*_a(f(x))):=f_b^*\circ A(e^*_a(f(x))e_a)$ defines a linear operator, which satisfies \eqref{inequality:sobolev.space}. Thus, it defines a bounded linear operator from $H^s(G)\rightarrow H^{s-m}$ \cite{ruzhansky2013global}.
Using the identification of the Sobolev spaces \eqref{eq:unitary.equivalence} and to the linear combination using the operator $B$, it can be seen that 
\begin{align}
A : (H^s(G)\otimes E_0)^K \longrightarrow (H^{s-m}(G)\otimes F_0)^K
\end{align}
is bounded.
\end{proof}

%Being invariant under left and right translations, $\mathcal L$ is a central operator and its group Fourier transform is scalar:

\begin{definition}[Symbol classes $\Sigma^m(E,F) $]
Let $m \in \R$. We denote $\sigma_A\in\Sigma_0^m(E,F)$ if the singular support of the map $y \mapsto K_A(x, y)$ is in $\{e\}$ and if 
\begin{align}
\left\Vert\triangle_\lambda^\alpha X^\beta\sigma_A(\lambda, x)\right\Vert_{op}
\leq C_{A \alpha \beta m}\langle\lambda\rangle^{m-|\alpha|}
\end{align}
for all $x \in G$, all multi-indices $\alpha, \beta$ and all $\lambda\in$Rep$(G)$. Moreover, we say that $\sigma_A\in\Sigma_{k+1}^m(E,F)$ if and only if
\begin{align}
\sigma_A \in\Sigma_k^m(E,F) \\ 
\sigma_{\partial_j}\sigma_A - \sigma_A\sigma_{\partial_j}\in\Sigma_k^m(E,F)\\ 
\left(\triangle_\lambda^\gamma \sigma_A\right) \sigma_{\partial_j} \in \Sigma_k^{m+1-|\gamma|}(E,F) %\\
\end{align}
for all $|\gamma|>0$ and $1 \leq j \leq \dim(G)$. Let 
\begin{align}
\Sigma^m(E,F) :=\bigcap_{k=0}^{\infty} \Sigma_k^m(E,F).
\end{align}
\end{definition}

\begin{theorem}
Suppose $G$ is a compact Lie group with a closed subgroup $K$ and $m \in \R$. 
If $E$ and $F$ are homogeneous vector bundles over $G/K$ associated to irreducible representations 
$E_0$ and $F_0$ of $K$, then, $A \in \Psi^m(E,F)$ if and only if $\sigma_A\in\Sigma^m(E,F)$.
\end{theorem}

\begin{proof}
%576 (583 of the pdf).
By repeating the argument of Theorem \ref{theorem:boundedness.in.sobolev.space}, this theorem reduces to Theorem 10.9.6 in \cite{ruzhansky2009pseudo}.
\end{proof}

\section{\bf Application and parametrics}
\label{sec:parametrix}

In this section, we are concerned with the case $E=F$, so we set 
$\Psi^m(E):=\Psi^m(E,E)$

\begin{definition}
Let $A\in \Psi^m(E)$ and $z\in\C$. $B_z \in\Psi^{-m}(E)$ will be
called a resolvent parametric of $A$ if it is a parametric of $A - z$ 
(that is, $B_z(A  - z)$ differs from the identity up to a smoothing operator).
\end{definition}

In this section, we present an explicit computation of the asymptotic expansion of the intrinsic symbol of a resolvent parametric of $A$ in terms of the representation of $G$ and the terms in the asymptotic expansion of the intrinsic symbol of $A$.

\begin{theorem}[Parametrix]%Theorem 10.9.10
\label{theorem:parametrix}
Let $\sigma_{A_j} \in \Sigma^{m - j}(E)$, and set 
\begin{align}
\sigma_A(\lambda, x) \sim \sum_{j=0}^\infty\sigma_{A_j}(\lambda, x).
\end{align}
Assume that $\sigma_{A_0}(\lambda, x) = \sigma_{B_0}(\lambda, x)^{-1}$ is an invertible matrix for every $x\in G$ and $\lambda\in$Rep$(G)$, and that $B_0=\mathrm{Op}\left(\sigma_{B_0}\right) \in \Psi^{- m}(E)$. 
Then, there exists $\sigma_B \in \Sigma^{-m}(E)$ such that $I-B A$ and $I-A B$ are smoothing operators. Moreover, 
\begin{align}
\sigma_B(\lambda, x) \sim \sum_{k=0}^\infty\sigma_{B_k}(\lambda, x)
\end{align}
where the operators $B_k \in \Psi^{-m-k}(E)$ are determined by the recursion
\begin{align}
\label{eq:composition.formula}
\sigma_{B_N}(\lambda, x) = - \sigma_{B_0}(\lambda, x) \sum_{k=0}^{N - 1} \sum_{j=0}^{N-k} \sum_{|\gamma|=N-j-k} \frac{1}{\gamma !}\triangle_{\lambda}^{\gamma} \sigma_{B_k}(\lambda, x) X^{\gamma} \sigma_{A_j}(\lambda, x).
\end{align}
\end{theorem}

\begin{proof}
%page578(584 of the pdf file).
If $\sigma_I\sim\sigma_{BA}$ for some $\sigma_B\sim\sum_{k=0}^\infty\sigma_{B_k}$, then by Theorem \ref{theorem:composition.formula} we have
\begin{align}
\begin{split}
I_{\dim\lambda}\otimes I_{\dim E_0} = \sigma_I(\lambda,x) & \sim \sigma_{BA}(\lambda,x) \\ 
& \sim \sum_{\gamma \geq 0} \frac{1}{\gamma !}\left(\triangle_\lambda^\gamma\sigma_B(\lambda,x)\right)X^\gamma\sigma_A(\lambda,x) \\ 
& \sim \sum_{\gamma \geq 0} \frac{1}{\gamma !}\left(\triangle_\lambda^\gamma\sum_{k=0}^\infty\sigma_{B_k}(\lambda,x)\right)X^\gamma\sum_{j=0}^{\infty} \sigma_{A_j}(\lambda,x).
\end{split}
\end{align}
We now find $\sigma_{B_k}$. Note that $I_{\dim\lambda}\otimes I_{\dim E_0} = \sigma_{B_0}(\lambda,x)\sigma_{A_0}(\lambda,x)$, and for $|\gamma|\geq1$ we may suppose that
\begin{align}
\sum_{|\gamma|=N-j-k} \frac{1}{\gamma !}\left(\triangle_{\lambda}^{\gamma} \sigma_{B_k}(\lambda, x)\right) X^\gamma \sigma_{A_j}(\lambda, x) = 0.
\end{align}
Then, \eqref{eq:composition.formula} provides the solution to these equations, and it can be easily verified that $\sigma_{B_N}\in\Sigma^{-m-N}(G)$ by induction on $N$ after noting that $B_0\in\Psi^{-m}(E)$. 
Thus, $\sigma_B\sim\sum_{k=0}^\infty\sigma_{B_k}$. Finally, notice that $\sigma_{I_{\dim\lambda}\otimes I_{\dim E_0}} \sim\sigma_{B A}$
\end{proof}

\section{\bf Example: the fiberation $\T\rightarrow SU(2)$}
\label{sec:example.T.su(2)}

We now present examples of symbols as aforementioned in Section \ref{subsec:homogeneous.vector.bundle}. For example, suppose $K  = \T \subset SU(2) = G$ and $E_n = \C^{\otimes n}\in\widehat{\T} \cong\Z$. In fact, $E_n \cong \Span\{e^{2\pi in t}:t\in\R\}$.
Note that $G/K=\C P^1 = S^2$.  Then, the homogeneous vector bundles associated to these representations are all line bundles $\mathcal O(n) := G\times_K\C^{\otimes n}$ and $\mathcal O(-n) := \mathcal O(n)^*$. More concretely, the sections $\Gamma\left(\mathcal O(n)\right)$ of the bundle $O(n)$ are given by the functions $f:G\rightarrow E_n$ such that 
\begin{align}
f
\begin{pmatrix}
e^{2\pi i t}\alpha & -e^{-2\pi i t}\bar \beta\\
e^{2\pi i t}\beta & e^{-2\pi i t}\bar\alpha
\end{pmatrix}
	& = 
e^{2\pi in t}f
\begin{pmatrix}
\alpha & -\bar \beta\\
\beta & \bar\alpha
\end{pmatrix}
\end{align}
In other words, 
\begin{align}
f(e^{2\pi it}\alpha, e^{2\pi it}\beta, e^{-2\pi it}\bar\alpha, e^{-2\pi it}\bar\beta)
=
e^{2\pi int}
f(\alpha, \beta, \bar\alpha,\bar\beta)
\end{align}
for all $t\in\R/\Z$. That is, 
\begin{align*}
f(\alpha, \beta, \bar\alpha,\bar\beta) = \alpha^{(n-k)}\beta^k\sum c_{\ell m p}(\alpha\bar\alpha)^\ell(\alpha\bar\beta)^m(\bar\alpha\beta)^p(\bar\beta\beta)^q.
\end{align*}
The sum $\sum c_{\ell m p}(\alpha\bar\alpha)^\ell(\alpha\bar\beta)^m(\bar\alpha\beta)^p$ corresponds to functions in $C^\infty(G/K)$.

We now compute the action of $\mathfrak{su}_\C(2)$ on the sections of $O(n)$. Let $\left\{ H, X,Y\right\}$ be the basis used in \eqref{eq:pauli.matrices}. Here, 
\[
\exp(tA) = e^{2\pi itA} := \sum \frac{(2\pi itA)^k}{k!}
\] 
for all $A\in\mathfrak{su}_\C(2)$. 
Their exponential is given by
\begin{align*}
\exp(tH) = &
\begin{pmatrix}
e^{2\pi it} & 0\\
0 & e^{-2\pi it}
\end{pmatrix},\quad 
\exp(tX) = 
\begin{pmatrix}
\cos(2\pi t) & \sin(2\pi t)\\
\sin(2\pi t) & \cos(2\pi t)
\end{pmatrix}\\
&\exp(tY) = 
\begin{pmatrix}
\cos(2\pi t) & \sin(2\pi t)\\
-\sin(2\pi t) & \cos(2\pi t)
\end{pmatrix}
.
\end{align*}
We can compute their actions on the sections. In fact, it is enough to compute their actions on $\alpha^{(n-k)}\beta^k$, so let $f(\alpha,\bar\alpha,\beta,\bar\beta)$.
		\begin{align}
	\begin{split}
H\cdot f(\alpha,\bar\alpha,\beta,\bar\beta)	& = \frac{d}{dt}\Big\vert_{t=0}f(e^{2\pi it}\alpha, e^{2\pi it}\beta, e^{-2\pi it}\bar\alpha, e^{-2\pi it}\bar\beta)\\
				& = \frac{d}{dt}\Big\vert_{t=0} e^{ - 2\pi int} f(\alpha,\bar\alpha,\beta,\bar\beta) \\
				& = ( - 2\pi n~i) f(\alpha,\bar\alpha,\beta,\bar\beta)
	\end{split}
		\end{align}
	\begin{align}
		\begin{split}
X\cdot f(\alpha,\bar\alpha,\beta,\bar\beta) & = (2\pi i)  f_*(\frac{d}{dt})_{(\bar\beta,\beta,\bar\alpha,\alpha)}
		\end{split}
	\end{align}
and
	\begin{align}
		\begin{split}
Y\cdot f(\alpha,\bar\alpha,\beta,\bar\beta)	& = (2\pi i) f_*(\frac{d}{dt})_{(\bar\beta,\beta,\bar\alpha,-\alpha)}\\
		\end{split}
	\end{align}
and the symbol $\sigma(\ell,x)$ of $A\in\mathfrak{su}(2)$ is given by 
\begin{align}
\sigma_A(\xi_\ell,x) 		& = \left(\xi_\ell\right)_*A\\
				\nonumber
					& = A\otimes \underset{2\ell~\mathrm{times}}{\underbrace{I\otimes\cdots \otimes I}} +I\otimes A\otimes I\otimes\cdots \otimes I + \cdots 
					+ I\otimes \cdots \otimes I\otimes A
\end{align}
where $\xi_\ell:=\operatorname{Sym}^{(2\ell + 1)}(\lambda)$ in \eqref{eq:representation.of.su(2)}, $\ell\in\frac12\N$. With the above notation, then, the action of a vector vector field $A\in\mathfrak{su}(2)$ on the sections of the homogeneous vector bundle associated to the representation $E_n\in\widehat\T=\Z$ can be written as 
\begin{align}
\label{formula:symbol.of.vector.fields}
A\cdot f(\alpha, \beta, \bar\alpha, \bar\beta) = \sum_{\ell\in\frac{1}{2}\N}(2\ell + 1)\Tr\left(\xi_\ell(\alpha, \beta, \bar\alpha, \bar\beta)
\left(\xi_\ell\right)_*A\widehat f(\ell)\right)
\end{align}
for $f\in\Gamma(\mathcal O(n))$ and the Haar measure on $SU(2)$ is given by 
\[
f\mapsto \frac1{2\pi^2}\int_{0}^1\int_0^1\int_{0}^{\pi/2} f(\alpha,\bar\alpha,\beta,\bar\beta)\sin\eta\cos\eta d\eta d\xi_1 d\xi_2
\]
where we identified 
\begin{align*}
\alpha = e^{2\pi i\xi_1}\sin\eta\\
\beta = e^{2\pi i\xi_2}\cos\eta.
\end{align*}

\section{\bf Concluding remarks and functional calculus}
\label{sec:remarks}

The main aim of this section is to provide a guideline of a derivation of the asymptotic expansions. 
%In particular, the use of $f(z)=e^{-tz}$ leads directly to another derivation of the Seeley-Minakshisundaram-Pleijel expansion where$A$ is a positive elliptic operator of order $m$. 
Note that each summand in \eqref{formula:composition} defines bilinear maps $P_j(\sigma,\tau)$ for $\sigma\in\Sigma^m(E)$ and $\tau\in\Sigma^{m'}(E)$
\[
P_j:\Sigma^m(E)\times\Sigma^{m'}(E)\longrightarrow\Sigma^{m+m' - j}(E)
\]
where $P(\sigma,\tau) := \sum_{j\leq m}P_j(\sigma,\tau)$ is the composition of symbols operators corresponding to $\sigma$ and $\tau$, respectively, computed using 
\eqref{formula:composition}. 

Suppose $\sigma\in \Sigma^m(E)$ such that $\sigma^{-1}\in\Sigma^{m'}(E)$ whereby $\sigma^{-1}$ denotes the inverse of $\sigma$ mod $\Sigma^{-\infty}$. Define
\begin{align}
\label{formula:inverse}
Q_0(\sigma) = \sigma^{-1},\quad Q_j(\sigma) = - \sigma^{-1} \sum_{r=1}^j P_r(\sigma, Q_{j-r}(\sigma)).
\end{align}
Induction shows that if $m + m'<1$, then
\[
Q_j(\sigma)\in\Sigma^{m' - j + j(m+m')}(E)
\]
and
\[
Q(\sigma) : = \sum_{j\geq0}Q_j(\sigma)
\]
defines an element of $\Sigma^{m'}(E)$. Moreover, \eqref{formula:inverse} and that $\sum_{r+s\leq N}P_r(\sigma, Q_s(\sigma)) = 1$ implies that 
\begin{align}
P(\sigma,Q(\sigma)) = 1.
\end{align}
This proves that if $A\in\Psi^m(E)$ and $\sigma_A^{-1}\in\Sigma^{m'}(E)$ with $m + m'<1$, then $A$ has a right inverse modulo $\Psi^{-\infty}(E)$, which belongs to $\Psi^{m'}(E)$ whose symbol is $Q(\sigma_A)$. However, a similar argument shows that $A$ also has a left inverse. Thus, we have proved the following improvement of Theorem \ref{theorem:parametrix}.

\begin{theorem}
If $A\in\Psi^m(E)$ and $\sigma^{-1}_A \in \Sigma^{m'}(E)$ with $m + m' <1$, then $A$ has a two-sided inverse in $\Psi^{m'}(E)$ whose symbol is $Q(\sigma_A)$ modulo $\Sigma^{-\infty}(E)$.
\end{theorem}

Using the above theorem, we expect that for an operator $A$ in the above class of operators whose symbol $\sigma_A$ has purely discrete spectrum $\Spec(\sigma_A)$ and an analytic function $f$ on $\Spec(\sigma_A)$, 
\begin{align}
\label{equation:cauchy.integration}
\sigma_{f(A)}	& = - \frac1{2\pi i}\int_\gamma f(\xi)Q(\sigma - \xi)d\xi\\
\nonumber
			& = f(\sigma) + \sum_{k=1}^\infty\sum_{l=2}^{2k} \frac{(-1)^{k+1}}{2\pi i}\int_\gamma f(\xi)Q_{k,l}(\sigma - \xi)d\xi
\end{align}

Since $(\sigma - z)^{-1}$ can be expressed as 
\begin{align}
(\sigma - z)^{-1} = \sum_{z_\alpha\in\Spec(\sigma_A)}\sigma_\alpha(z_\alpha - z)^{-1}.
\end{align}
Then, 
\begin{align}
Q_{k,l}(\sigma - z) = \sum_{\alpha_0,\alpha_1,\ldots,\alpha_k}\prod_{j=0}^k(z_{\alpha_j} - z)^{-1}Q_{k,l}(\sigma:\sigma_{\alpha_0},\ldots,\sigma_{\alpha_l})
\end{align}
where $Q_{k,l}(\sigma:\sigma_{\alpha_0},\ldots,\sigma_{\alpha_k})$ means the $j$th $\sigma^{-1}$ term in $Q_{k,l}(\sigma)$ is replaced by $\sigma_{\alpha_j}$. 
We can evaluate the integrals in \eqref{equation:cauchy.integration}.
\begin{align}
\label{formula:cauchy.integral.2}
\frac{(-1)^{k+1}}{2\pi i}\int_\gamma f(\xi)\prod_{j=0}^k(\zeta_j-\xi)^{-1}d\xi = \sum_{l=0}^kf(\zeta_l)\prod_{a\neq l}(\zeta_a - \zeta_l)^{-1}
\end{align}
for distinct $\zeta_l$. If $\zeta_l$ are not distinct, the above summation can be modified by replacing $f$ by its derivatives.
Thus, we denote \eqref{formula:cauchy.integral.2} by
\begin{align}
\frac1{k!}f^{(k)}(\zeta_0,\ldots,\zeta_k). %:= \sum_{l=0}^kf(\zeta_l)\prod_{a\neq l}(\zeta_a - \zeta_l)^{-1}.
\end{align}
With this notation, the integral in \eqref{equation:cauchy.integration} becomes
\begin{align}
\label{eq:funtional.calculus}
f(\sigma) + \sum_{k=1}^\infty\sum_{l=2}^{2k} \frac1{l!}f^{(l)}(\zeta_0,\ldots,\zeta_l) Q_{k,l}(\sigma:\sigma_{\alpha_0},\ldots,\sigma_{\alpha_l})
\end{align}

Indeed, if $f$ is analytic and $A\in\Psi^0(E)$, then a modification of \cite[Theorem 4.1]{widom1980complete} shows that $f(A)\in\Psi^0(E)$ and $f(\sigma_A)=\sigma_{f(A)}$ where the symbol $\sigma_{f(A)}$ is defined by the formula \eqref{eq:funtional.calculus}. However, the computation of explicit formula is very difficult, unless $f$ and $A$ are known explicitly.

\bibliographystyle{abbrv}
\bibliography{symbolic_calculus_for_homogeneous_vector_bundle.bbl}

\end{document}